\documentclass{amsart}
\usepackage{latexsym}
\usepackage{amssymb}
\usepackage{amsmath}
\usepackage{amsfonts}

\def\R{\mathbb{R}}

\newcommand{\bmat}{\left[\begin{matrix}}
\newcommand{\emat}{\end{matrix}\right]}
\usepackage[latin1]{inputenc} 
\usepackage{amsthm}
\usepackage{amsxtra}
\usepackage{amscd}
\usepackage{setspace}   
\usepackage{mathrsfs}   

\newtheorem{theorem}{Theorem}
\newtheorem{proposition}[theorem]{Proposition}
\newtheorem{lemma}[theorem]{Lemma}

\newtheorem*{corollary}{Corollary}
\numberwithin{theorem}{section}
\numberwithin{equation}{section}

\theoremstyle{remark}

\theoremstyle{definition}

\newcommand{\Z}{\mathbb{Z}}

\DeclareMathOperator{\diam}{diam}

\newcommand{\GL}{\mathrm{GL}}
\newcommand{\SL}{\mathrm{SL}}

\newcommand{\diag}{\mathrm{diag}}

\newcommand{\vb}{\mathbf{b}}
\newcommand{\vc}{\mathbf{c}}
\newcommand{\vk}{\mathbf{k}}

\newcommand{\vm}{\mathbf{m}}
\newcommand{\vn}{\mathbf{n}}
\newcommand{\vx}{\mathbf{x}}
\newcommand{\vy}{\mathbf{y}}

\newcommand{\vu}{\mathbf{u}}
\newcommand{\vv}{\mathbf{v}}
\newcommand{\vw}{\mathbf{w}}

\usepackage{enumerate}
\usepackage{hyperref}

\newcommand{\vbeta}{\boldsymbol\beta}

\newcounter{daggerfootnote}

    \title{Bounds on the gaps in the fractional parts of a linear form}
    \author[Seungki Kim]{Seungki Kim$^\dagger$}
    \thanks{$\dagger$ University of Cincinnati, 2600 Clifton Ave., Cincinnati OH 45221, United States. \\ e-mail: seungki.math@gmail.com}

\begin{document}
\maketitle

\begin{abstract}
We provide bounds on the sizes of the gaps --- defined broadly --- in the set $\{k_1\vbeta_1 + \ldots + k_n\vbeta_n \mbox{ (mod 1)} : k_i \in \Z \cap (0,Q^\frac{1}{n}]\}$ for generic $\vbeta_1, \ldots, \vbeta_n \in \R^m$ and all sufficiently large $Q$. We also introduce a related problem in Diophantine approximation, which we believe is of independent interest.
\end{abstract}







\section{Introduction and the statements of results}

\subsection{A brief introduction}

The three-gap theorem states that for any $\alpha \in \R$ and $Q > 0$, the set
\begin{equation*}
K(\alpha; Q) = \{k\alpha \mbox{ (mod 1)} : k \in \Z \cap (0,Q]\}
\end{equation*}
partitions $\R/\Z$ into intervals of at most three distinct lengths. This simple and attractive statement has caught the attention of a number of mathematicians for several decades, and continues to stimulate new research to this day (see e.g. \cite{MSthree} and the references therein; see also \cite{DH23, HM20, HM22}).

The purpose of the present paper is to propose and explore the following simple, natural curiosity: what can we say about those lengths? More precisely, can we prove any asymptotics or estimates as $Q \rightarrow \infty$?

The part of the literature most closely related to this question seems to be that of the discrepancy function of a Kronecker sequence. Given $\vbeta \in \R^m$ and the associated Kronecker sequence
\begin{equation*}
K(\vbeta;Q) = \{k\vbeta \mbox{ (mod 1)} : k \in \Z \cap (0,Q]\},
\end{equation*}
the \emph{discrepancy function} $\Delta$ is defined to be
\begin{equation*}
\Delta(\vbeta;Q) = \max_{1 \leq q \leq Q \atop 0 \leq x_1, \ldots, x_m \leq 1} \Big| qx_1 \cdots x_m - \big|K(\vbeta;q) \cap \prod_{1 \leq i \leq m} [0,x_i)\big| \Big|.
\end{equation*}
Regarding the size of $\Delta$, there is a beautiful theorem of Beck \cite[Theorem 1]{Beck} stating that, for any positive increasing function $\varphi: \R \rightarrow \R$,
\begin{equation} \label{eq:Beck}
\Delta(\vbeta;Q) \ll (\log Q)^m \cdot \varphi(\log \log Q) \Leftrightarrow \sum_{j=1}^\infty \frac{1}{\varphi(j)} < \infty
\end{equation}
for almost every $\vbeta$.
As a crude application of this result, it is possible to prove, for example, that the maximal size of the ``gap'' --- let us provisionally define it as the volume of a hypercube in $\R^m/\Z^m$ not intersecting $K(\vbeta;Q)$ --- is at most of order
\begin{equation*}
Q^{-1} \cdot (\log Q)^m \cdot \varphi(\log \log Q),
\end{equation*}
for any positive increasing $\varphi$ satisfying the right-hand side of \eqref{eq:Beck}. It seems this is also the best bound obtainable this way, since \eqref{eq:Beck} does not yield meaningful information when $Qx_1 \cdots x_m = o\left((\log Q)^m \cdot \varphi(\log \log Q)\right)$.

In this paper, we show, among other things, that the above bound can be improved to
\begin{equation*}
Q^{-1} \cdot \log Q \cdot \varphi(\log \log Q).
\end{equation*}
In fact, we consider our topic in a much broader generality, in terms of both the sequences $K$ and the notion of the gap. We provide the details in the next section.

\subsection{A bit of setup, and the main results} \label{sec:intro_setup}

Much of our setup closely follows that of the recent works related to the three-gap theorem, such as Marklof-Str\"ombergsson \cite{MSthree} and Haynes-Marklof \cite{HM20,HM22}. The definitions below are adapted from these works, and then extended to  accommodate general linear forms and various notions of gaps.

Fix integers $m,n > 0$. For an $m \times n$ matrix
\begin{equation*}
B = \begin{pmatrix} | & & | \\ \vbeta_1 & \cdots & \vbeta_n \\ | & & | \end{pmatrix}
\end{equation*}
and a positive real number $Q$, write
\begin{equation*}
K(B;Q) = \{B\vk \mbox{ (mod 1)} : \vk \in \Z^n \cap (0, Q^\frac{1}{n}]^n\}.
\end{equation*}
Choose also $p \in [1,\infty]$ and an open subset $S$ of the unit sphere in $\R^m$. The variables $m,n, p, S$ may be considered fixed throughout this paper.

Let $d = m+n$ and $\mathrm{cone}(S) = \{c\vu : \vu \in S, c > 0\}$. For $T \subseteq \R^m$, write
\begin{equation*}
{\inf}_p^S T := \inf\{ \|\vy\|_p : \vy \in T \cap \mathrm{cone}(S) \},
\end{equation*}
the infimum of the $L^p$-norm of the vectors in $T$ in the direction of $S$.
In addition, for $\vv \in \R^m$, define
\begin{equation} \label{eq:gap2}
\gamma_{(B;Q)}(\vv) = {\inf}_p^S\{ B\vn + \vm - \vv \in \R^m : (\vn, \vm) \in \Z^d, \vn \in (0, Q^\frac{1}{n}]^n\},
\end{equation}
the distance from $\vv$ (mod $1$) to the nearest point in $K(B;Q)$ in the direction of $S$ with respect to the $L^p$-norm.

Our question is then whether we can give any size estimate on $\gamma_{(B;Q)}(\vv)$ as $Q \rightarrow \infty$. Of particular interest may be when $\vv$ is restricted to the elements of $K(B;Q)$. We are able to prove the following bounds.

\begin{theorem} \label{thm:main1}
Suppose $f: \R_{>0} \rightarrow \R_{>0}$ is an increasing function such that
\begin{equation*}
\sum_{k=1}^\infty \frac{1}{f(k)} < \infty.
\end{equation*}
Then the following statements hold.
\begin{enumerate}[(a)]
\item For almost all $m \times n$ matrix $B$, and all sufficiently large $Q > 0$,
\begin{equation*}
\sup_{\vv \in K(B;Q)} \gamma_{(B;Q)}(\vv) \leq \sup_{\vv \in \R^m} \gamma_{(B;Q)}(\vv) < Q^{-\frac{1}{m}}\left(f(\log Q)\right)^\frac{1}{m}.
\end{equation*}

\item If we further assume that $f(x)$ grows at most polynomially in $x$, then also
\begin{equation*}
Q^{-\frac{1}{m}}\left(f(\log Q)\right)^{-\frac{1}{m^2}} < \sup_{\vv \in \R^m} \gamma_{(B;Q)}(\vv).
\end{equation*}

\item Under the assumptions of (a) and (b) above,
\begin{equation*}
Q^{-\frac{1}{m}}\left(f(\log Q)\right)^{-\frac{1}{m}} < \inf_{\vv \in K(B;Q)} \gamma_{(B;Q)}(\vv).
\end{equation*}
If in addition $S \cup (-S)$ covers the entire unit sphere, then
\begin{equation*}
\inf_{\vv \in K(B;Q)} \gamma_{(B;Q)}(\vv) \ll_{d,p} Q^{-\frac{1}{m}}.
\end{equation*}

\end{enumerate}
\end{theorem}

Our proof of the above bounds starts by translating the problem to the space of lattices \`a la Marklof-Str\"ombergsson \cite{MSthree}. Then methods in the geometry of numbers show that the size of $\gamma_{(B;Q)}$ are controlled by the successive minima of some $d$-dimensional lattices. The growth and decay of the successive minima are then in turn controlled by the zero-one law due to Kleinbock-Margulis \cite{KM98} (see also \cite[Proposition 1.11]{KS22}).

One may reasonably be curious about other bounds of the above kind. Following a suggestion of an anonymous referee, we attempted to give a sharper lower bound on $\sup_{\vv \in K(B;Q)} \gamma_{(B;Q)}(\vv)$ and an upper bound on $\inf_{\vv \in K(B;Q)} \gamma_{(B;Q)}(\vv)$, but the only success we had was with the latter case in which $S$ is rather large, as stated in Theorem \ref{thm:main1}(c) above; see the footnotes in Section 3 for brief comments on the difficulties involved. On the other hand, note that $\inf_{\vv \in \R^m} \gamma_{(B;Q)}(\vv) = 0$ trivially.

Theorem \ref{thm:main1} has an application to the following problem in Diophantine approximation, which we believe is of independent interest. In fact, this was the problem we started investigating at first, but we subsequently found that it is closely related to the gap sizes in $K(B;Q)$, which led us to Theorems \ref{thm:main1}.

\begin{corollary}
Let $B$ be an $m \times n$ real-valued matrix, and $N > 0$ be an integer. Consider the inequalities
\begin{equation} \label{eq:diophantine}
\|B\vn - \vm\|_\infty^m < N^{m}\varphi(Q) \mbox{ and } \|\vn\|_\infty^n < (N/2)^nQ,
\end{equation}
where $\vm \in \Z^m, \vn \in \Z^n$, and $\varphi:\R_{>0} \rightarrow \R_{>0}$ is a decreasing function.

Then for almost every $B$, and for all $N$ and all $Q \gg_{B,N} 0$, \eqref{eq:diophantine} has a solution $(\vm,\vn) \in \Z^d$ in \emph{every} congruence class of $\Z^d$ modulo $N$, if
\begin{equation*}
\varphi(Q) \geq Q^{-1}f(\log Q)
\end{equation*}
for $f$ as in the statement of Theorem \ref{thm:main1}.

On the other hand, if
\begin{equation*}
\varphi(Q) \leq Q^{-1}\left(f(\log Q)\right)^{-\frac{1}{m}}
\end{equation*}
for some $f$ as in Theorem \ref{thm:main1}(b), then for $Q \gg_{B,f} 0$ and $N \gg_Q 0$, \eqref{eq:diophantine} fails to have a solution $(\vm,\vn) \in \Z^d$ in some congruence class of $\Z^d$ modulo $N$.
\end{corollary}

It may be of interest to compare this statement with other recent works on Diophantine approximation with congruence conditions. Nesharim-R\"uhr-Shi \cite{NRS20} proved a refined version of the classical Khintchine-Groshev theorem in which the solutions are restricted to a fixed congruence class; Alam-Ghosh-Yu \cite{AGY21} improved it further, by providing an asymptotic formula on the number of such solutions. Both are results in what is sometimes referred to as \emph{asymptotic approximation}, whereas ours is a result in \emph{uniform approximation}. Of course, its main drawback is that it does not provide a strict zero-one law, reflecting the gap in our estimates of $\sup_{\vv \in \R^m} \gamma_{(B;Q)}(\vv)$.

\subsection{Organization}

We collect the results needed for later arguments in Section 2. Section 3 proves Theorem \ref{thm:main1} and its corollary.

\subsection{Acknowledgment}

We thank the anonymous referees for the various comments and suggestions that helped improve the present paper in a number of ways.

\section{Preliminaries}

\subsection{Setup}

We continue to use the notations defined in Section \ref{sec:intro_setup} throughout this paper. In this section, we adopt the idea of Marklof-Str\"ombergsson \cite{MSthree} to phrase our gap problem in terms of a geodesic flow on the space of lattices. The same framework also appears in Haynes-Marklof \cite{HM20,HM22} as well.

Write
\begin{equation*}
U_B = \begin{pmatrix} \mathrm{Id}_n &  \\ B & \mathrm{Id}_m \end{pmatrix},
A_Q = \begin{pmatrix} Q^{-\frac{1}{n}}\mathrm{Id}_n &  \\  & Q^{\frac{1}{m}}\mathrm{Id}_m \end{pmatrix}.
\end{equation*}
Then $\gamma_{(B;Q)}$ defined in \eqref{eq:gap2} can be reformulated as
\begin{align*}
\gamma_{(B;Q)}(\vv)
&= {\inf}_p^S\{ B\vn + \vm - \vv \in \R^m : (\vn, \vm) \in \Z^d, \vn \in (0, Q^\frac{1}{n}]^n \} \\
&= {\inf}_p^S\{ \vy - \vv \in \R^m : (\vx, \vy) \in U_B\Z^d, \vx \in (0, Q^\frac{1}{n}]^n \} \\
&= Q^{-\frac{1}{m}} {\inf}_p^S\{ \vy - Q^\frac{1}{m}\vv \in \R^m : (\vx, \vy) \in A_QU_B\Z^d, \vx \in (0, 1]^n \} \\
&= Q^{-\frac{1}{m}}\Phi(A_QU_B, Q^\frac{1}{m}\vv),
\end{align*}
where
\begin{equation*}
\Phi(M, \vw) = {\inf}_p^S\{ \vy - \vw \in \R^m : (\vx, \vy) \in M\Z^d, \vx \in (0,1]^n\}
\end{equation*}
for $M \in \GL(d,\R)$ and $\vw \in \R^m$. We leave open the possibility that $\Phi(M,\vw) = \infty$.

\subsection{Results from the geometry of numbers}
By the word \emph{lattice} we mean a finitely generated discrete $\Z$-submodule of a Euclidean space. Unless we mention otherwise, we assume the lattice is of the full rank, that is, its rank is the same as the dimension of its ambient space.
For a lattice $L \subseteq \R^d$ and $i = 1, 2, \ldots, d$, recall that its \emph{$i$-th successive minimum} $\lambda_i(L)$ is the smallest real number such that the elements of $L$ of length (with respect to the $L^2$-norm) less than or equal to $\lambda_i(L)$ span a subspace of dimension at least $i$. For a $d \times d$ matrix $M$, let us write $\lambda_i(M)$ for the $i$-th minimum of the lattice generated by the columns of $M$.

The following theorem plays a crucial role in bounding $\gamma_{(B;Q)}$ from above. 
It and the two lemmas that follow are widely known: in fact, it can be seen as a rephrasing of Babai's nearest plane algorithm (see e.g., \cite[Chapter 2.2]{Pre15} and the references therein); also, \cite[Lemma 2.1]{Str11} gives the identical statement, except that it differs from ours by an exponential factor in $d$. 
Here we provide a formulation that is convenient for our usage, and supply a proof for completeness.

\begin{theorem} \label{thm:existence}
Let $L \subseteq \R^d$ be a (full-rank) lattice. Then any set containing a closed ball of diameter (with respect to the $L^2$-norm) $\sqrt{d}\lambda_d(L)$ contains a point of $L$.
\end{theorem}

This is an immediate consequence of the two lemmas below.

\begin{lemma} \label{lemma:lattice1}
Every lattice $L \subseteq \R^d$ admits a covering by rectangles of diameter $\sqrt{d}\lambda_d(L)$.
\end{lemma}
\begin{proof}
Choose linearly independent $\vv_1, \ldots \vv_d \in L$ such that $\|\vv_i\|_2 = \lambda_i(L)$. The $\vv_i$'s span a sublattice $L'$ of $L$ of finite index. We will demonstrate a tiling of $\R^d$ by the $L'$-translates of the rectangles of diameter at most $\sqrt{d}\lambda_d(L)$, which will prove the lemma.

It is a well-known fact that the lattice generated by the column vectors of a $d \times d$ upper-triangular unipotent matrix
\begin{equation*}
\begin{pmatrix}
1 & * & *          & \cdots & * \\
   & 1 & *          & \cdots & *     \\
   &    & \ddots & \vdots & \vdots   \\
   &    &           &   1       &  * \\
   &    &           &           &  1
\end{pmatrix}
\end{equation*}
admits a tiling of $\R^d$ by the $d$-dimensional unit cube $[-1/2,1/2]^d$. One quick way to see this is as follows: denote by $\vc_i$ the $i$-th column vector of the above matrix, by $G$ the lattice generated by $\{\vc_1,\ldots,\vc_d\}$, and by $G'$ the (corank $1$) sublattice of $G$ generated by $\{\vc_1,\ldots,\vc_{d-1}\}$. Then since
\begin{equation*}
G = \coprod_{k \in \Z} (G' + k\vc_d),
\end{equation*}
and the last coordinate of $\vc_d$ is $1$, the problem reduces to tiling $\R^{d-1}$ by the $G'$-translates of the $(d-1)$-dimensional unit cube. One can now invoke induction to prove the claim.

Back to the lattice $L'$ generated by the $\vv_i$'s, the matrix whose $i$-th column is $\vv_i$ has the Iwasawa decomposition $KAN$, where $K \in \mathrm{O}(d,\R)$, $A = \diag(a_1,\ldots,a_d)$ with $0 < a_i \leq \|\vv_i\| = \lambda_i(L)$, and $N$ is upper-triangular unipotent. By the above discussion, the lattice generated by the columns of $N$ admits a tiling by the unit cube. Applying $KA$ to it, we see that $L'$ admits a tiling by a rectangle of diameter $\sqrt{a_1^2 + \ldots + a_d^2} \leq \sqrt{d}\lambda_d(L)$, as desired.
\end{proof}

\begin{lemma} \label{lemma:lattice2}
Let $L \subseteq \R^d$ be a lattice, and $C \subseteq \R^d$ be a convex, centrally symmetric closed set such that its $L$-translates cover $\R^d$. Then any $D \subseteq \R^d$ containing a closed ball of diameter $\diam(C)$ contains a point of $L$.
\end{lemma}
\begin{proof}
It suffices to assume $D$ is a closed ball of diameter $\diam(C)$. Suppose $L \cap D$ is empty. Then there exists an $\varepsilon > 0$ such that every point of $L$ is at least $\frac{1}{2}\diam(C) + \varepsilon$ units of distance away from the center of $D$. But this contradicts the assumption that the translates of $C$ by the points of $L$ cover $\R^d$.
\end{proof}

\begin{proof}[Proof of Theorem \ref{thm:existence}]
By Lemma \ref{lemma:lattice1}, $\R^d$ is covered by the $L$-translates of a rectangle $C$ of diameter $\sqrt{d}\lambda_d(L)$. Applying Lemma \ref{lemma:lattice2} with this $C$ completes the proof.
\end{proof}

The two lemmas below will be used in Section 4 to give a lower bound on $\sup_{\vv \in \R^m}\gamma_{(B;Q)}(\vv)$.

\begin{lemma} \label{prop:LLL}
Let $L \subseteq \R^d$ be a lattice. Then there exists a subspace $S_m \subseteq \R^d$ of dimension $m-1$ such that $L$ is contained in a disjoint union of the translates of $S_m$. Furthermore, the ($L^2$-)distance between any two of those translates is at least $c_1(d)\lambda_m(L)$, for a constant $c_1(d) > 0$ depending only on $d$.
\end{lemma}
\begin{proof}
We need the reduction theory on the space of lattices, for which we adopt the language and the results of \cite{LLL}; it is also an excellent reference for readers seeking more details on the subject.

Take a reduced basis $\{\vb_1, \ldots, \vb_d\}$ of $L$, whose Gram-Schmidt orthogonalization is given by $\{\vb_1^*, \ldots, \vb_d^*\}$ (see \cite[pp.516]{LLL} for both notions). We let $S_m = \mathrm{span}\{\vb_1,\ldots, \vb_{m-1}\}$. It is clear that $L$ is contained in the union of the translates of $S_m$ by the integer linear combinations of $\vb_m, \ldots, \vb_d$, and that any two translates are at least $\min_{m \leq i \leq d} \|\vb_i^*\|_2$ apart. But by \cite[(1.7)]{LLL} and the remark following \cite[Proposition 1.12]{LLL}, $\min_{m \leq i \leq d} \|\vb_i^*\|_2$ is bounded from below by $c_1(d)\lambda_m(L)$ for some $c_1(d)>0$.
\end{proof}

\begin{lemma} \label{prop:lower2}
Continue with the settings of Lemma \ref{prop:LLL}. Then there exists $c_2(d) > 0$ and $\vy \in \R^m$ such that $[0,c_2(d)\lambda_m(L)]^d + (\mathbf 0_n, \vy)$ is disjoint from any translate of $S_m$ by a point of $L$.
\end{lemma}
\begin{proof}
Let us take $c_2(d) = (1-(2n)^{-1})\sqrt{d}^{-1}c_1(d)$, say. Since $S_m$ is $(m-1)$-dimensional, it is possible to choose $\vy$ so that $[0,c_2(d)\lambda_m(L)]^d + (\mathbf 0_n, \vy)$ is disjoint from any fixed translate of $S_m$ by a point of $L$. If it is impossible to choose $\vy$ such that it is disjoint from any translate of $S_m$ by $L$, then for some $\vy$ it must intersect with two distinct such translates. But this contradicts Lemma \ref{prop:LLL}, which states that they are at least $c_1(d)\lambda_m(L)$ apart in distance.
\end{proof}


\subsection{Zero-one laws on diagonal flows}

Theorems 1.7 and 8.7, and Proposition 7.1 of Kleinbock and Margulis \cite{KM98} imply the following Borel-Cantelli type statement, paraphrased in the form convenient for our later use.

\begin{theorem}[Kleinbock-Margulis \cite{KM98}, see also Kim-Skenderi {\cite[Proposition 1.11]{KS22}}] \label{thm:0-1}
Suppose $f: \R_{>0} \rightarrow \R_{>0}$ is an increasing function, and $U \in \SL(d,\R)$. Then for $1 \leq i \leq d-1$,
\begin{equation*}
\sum_{k=1}^\infty \frac{1}{f(k)} < \infty \Leftrightarrow \mbox{ for almost every $U$, } \lambda_i(A_RU) > (f(\log R))^{-\frac{1}{di}} \mbox{ for $R \gg_U 0$}.
\end{equation*}
Also, for $2 \leq j \leq d$,
\begin{equation*}
\sum_{k=1}^\infty \frac{1}{f(k)} < \infty \Leftrightarrow \mbox{ for almost every $U$, } \lambda_j(A_RU) < (f(\log R))^{\frac{1}{d(d-j+1)}} \mbox{ for $R \gg_U 0$}.
\end{equation*}

The same statements continue to hold, if $U$ is replaced by $U_B$ for an $m \times n$ matrix $B$, and ``almost every $U$'' is replaced by ``almost every $B$.''

\end{theorem}


\section{Proof of Theorem \ref{thm:main1} and its corollary}

\subsection{Proof of (a)} \label{sec:upper}

Take $U \in \SL(d,\R)$ and $R \in \R_{> 0}$, and let
\begin{equation*}
P = \left(\left(1+n^{-1}\right)\sqrt{d}\lambda_d(A_RU)\right)^n
\end{equation*}
(any choice slightly greater than $(\sqrt{d}\lambda_d(A_RU))^n$ will do). We determine $Q$ by $Q = PR$.

By Theorem \ref{thm:existence}, an arbitrary translate of $H = (0,P^\frac{1}{n}]^d \subseteq \R^d$ contains a point of $A_RU\Z^d$. This implies that an arbitrary translate of $A_PH = (0,1]^n \times (0, P^{\frac{d}{mn}}]^m$ contains a point of $A_PA_RU\Z^d = A_QU\Z^d$; equivalently, any translate of $A_PH$ contains a point of any translate of $A_QU\Z^d$. With this in mind, consider the set
\begin{equation*}
(0,1]^n \times \left( B_p^m(rP^\frac{d}{mn}) \cap \mathrm{cone}(S) \right)
\end{equation*}
for some $r > 0$, where $B_p^m(t) \subseteq \R^m$ here denotes the $L^p$-ball of radius $t$ at origin. If $r$ is sufficiently large, depending only on $p$ and $S$ (for instance, if $p = \infty$ and $S$ is the entire unit sphere, then we can take $r = 1/2$) then this set contains a translate of $A_PH$, and thus contains a point $A_QU\Z^d - (\mathbf 0_n, \vw)$ for any $\vw \in \R^m$.
Therefore
\begin{equation} \label{eq:upper}
\Phi(A_QU, \vw) < rP^\frac{d}{mn}
\end{equation}
for any $\vw \in \R^m$.

We specialize to the case $U = U_B$ and $\vw = Q^{\frac{1}{m}}\vv$. By Theorem \ref{thm:0-1}, for every increasing function $f$ satisfying
\begin{equation*}
\sum_{k=1}^\infty \frac{1}{f(k)} < \infty,
\end{equation*}
it follows that, for almost every $B$ and all sufficiently large $R$,
\begin{equation*}
\Phi(A_QU_B, Q^{\frac{1}{m}}\vv) < rP^\frac{d}{mn} < f(\log R)^\frac{1}{m},
\end{equation*}
by applying Theorem \ref{thm:0-1} with an appropriate constant multiple of $f$ if necessary.

At this point, observe that $P \geq 1$, since $\lambda_d(L) \geq 1$ for any lattice $L \subseteq \R^d$ of covolume $1$. It follows that $R \rightarrow \infty$ if and only if $Q \rightarrow \infty$, since $Q$ is bounded from both sides by functions increasing in $R$, and that
\begin{equation*}
f(\log R) = f(\log Q - \log P) \leq f(\log Q).
\end{equation*}

Summarizing the discussion so far, we conclude that
\begin{equation*}
\Phi(A_QU_B, Q^{\frac{1}{m}}\vv) < f(\log Q)^\frac{1}{m},
\end{equation*}
or equivalently
\begin{equation*}
\gamma_{(B;Q)}(\vv) < Q^{-\frac{1}{m}}\left(f(\log Q)\right)^{\frac{1}{m}},
\end{equation*}
for any $\vv \in \R^m$ all sufficiently large $Q$, depending on $p, S$, and $B$.

\subsection{Proof of (b)}

As earlier, set $R = QP^{-1} \in \R_{>0}$, but this time define
\begin{equation*}
P = \left(c_2(d)\lambda_m(A_RU_B)\right)^n, 
\end{equation*}
where $c_2(d)$ is as in Lemma \ref{prop:lower2}. Lemma \ref{prop:lower2} implies that $A_RU_B\Z^d$ has no vectors inside $(0,P^\frac{1}{n}]^d + (\mathbf 0_n, \vy)$ for some $\vy \in \R^m$, which implies in turn that $A_QU_B\Z^d$ has no vectors inside $(0,1]^n \times (0, P^\frac{d}{mn}]^m + (\mathbf 0_n, P^\frac{1}{m}\vy)$. This implies that
\begin{equation*}
\Phi(A_QU_B, \vw) > rP^\frac{d}{mn}
\end{equation*}
for some $\vw \in \R^m$ and $r > 0$ depending on $p$.\footnote{If some point of $K(B;Q)$ is close enough --- of distance $O(P^\frac{d}{mn})$ --- to $Q^{-\frac{1}{m}}\vw$ (mod $1$), one could demonstrate a lower bound on $\sup_{\vv \in K(B;Q)} \gamma_{(B;Q)}(\vv)$. However, it turns out that $K(B;Q)$ is not always dense enough for this.}

Invoking Theorem \ref{thm:0-1}, for every increasing function $f$ satisfying
\begin{equation*}
\sum_{k=1}^\infty \frac{1}{f(k)} < \infty,
\end{equation*}
it holds that, for almost every $B$ and all sufficiently large $R$,
\begin{equation} \label{eq:lower1}
\Phi(A_QU_B, \vw) > rP^\frac{d}{mn} > (f(\log R))^{-\frac{1}{m^2}}.
\end{equation}

We now use the assumption that $f(x) < x^k$ for some $k > 0$ and all sufficiently large $x$. Then for almost every $B$ and all sufficiently large $R$
\begin{equation} \label{eq:lower2}
(\log R)^{-\frac{k}{m^2}} < P < (\log R)^{\frac{k}{m^2}}
\end{equation}
holds: the upper bound can be derived from Theorem \ref{thm:0-1} for $m \neq 1$, and from Minkowski's bound $\lambda_1(L) \leq \sqrt{d}(\det L)^{1/d}$ for $m = 1$. Therefore, $Q = RP$ is bounded from both sides by increasing functions in $R$, and thus $R \rightarrow \infty$ if and only if $Q \rightarrow \infty$. In addition, \eqref{eq:lower2} makes it clear that $\log R < 2\log Q$ for all sufficiently large $Q$. Therefore, with $\vw = Q^{\frac{1}{m}}\vv$,
\begin{equation*}
\gamma_{(B;Q)}(\vv) = Q^{-\frac{1}{m}}\Phi(A_QU_B, \vw) > Q^{-\frac{1}{m}}\left(f(2\log Q)\right)^{-\frac{1}{m^2}}
\end{equation*}
for all sufficiently large $Q$, depending on $p$ and $B$. Since
\begin{equation*}
\sum_{k=1}^\infty \frac{1}{f(k)} < \infty \mbox{ if and only if } \sum_{k=1}^\infty \frac{1}{f(2k)} < \infty,
\end{equation*}
the desired inequality follows.

\subsection{Proof of (c)} \label{sec:lower}

The argument for the lower bound is only slightly different from that of the previous section. Again we write $R = QP^{-1}$, but this time
\begin{equation*}
P = \left((1-(2n)^{-1})\sqrt{d}^{-1}\lambda_1(A_RU_B)\right)^n.
\end{equation*}

By construction, $A_RU_B\Z^d$ has no nonzero vector inside the box $H = \{\vx \in \R^d : \|\vx\|_\infty \leq P^\frac{1}{n}\}$. Therefore,  $A_QU_B\Z^d$ has no nonzero vectors inside $A_PH = [-1,1]^n \times [-P^\frac{d}{mn}, P^\frac{d}{mn}]^m$. This implies that
\begin{equation*}
\Phi(A_QU_B, Q^\frac{1}{m}\vv) > rP^\frac{d}{mn}
\end{equation*}
for any $\vv \in K(B;Q)$, and some $r > 0$ depending on $p$. This is because, if $\vv = B\vk$ for some $\vk \in \Z^n \cap (0,Q^\frac{1}{n}]^n$, then we have
\begin{equation*}
\Phi(A_QU_B, Q^\frac{1}{m}\vv) = {\inf}^S_p\{\vy \in \R^m : (\vx,\vy) \in A_QU_B\Z^d, \vx \in (0,1]^n - Q^{-\frac{1}{n}}\vk\}.
\end{equation*}

By Theorem \ref{thm:0-1}, for every increasing function $f$ satisfying
\begin{equation*}
\sum_{k=1}^\infty \frac{1}{f(k)} < \infty,
\end{equation*}
it holds that, for almost every $B$ and all sufficiently large $R$,
\begin{equation*}
\Phi(A_QU_B, Q^\frac{1}{m}\vv) > rP^\frac{d}{mn} > (f(\log R))^{-\frac{1}{m}}.
\end{equation*}

The rest of the proof now proceeds as in part (b) earlier, with \eqref{eq:lower1} replaced by the above inequality, and \eqref{eq:lower2} replaced by
\begin{equation*}
(\log R)^{-\frac{k}{m}} \leq P \leq (1-(2n)^{-1})^n.
\end{equation*}

For the upper bound, notice that if we set instead
\begin{equation*}
P = \left(\lambda_1(A_RU_B)\right)^n,
\end{equation*}
then $A_RU_B\Z^d$ must have a nonzero vector inside the box $H$ defined above. Therefore $A_QU_B\Z^d$ must have a nonzero vector inside $A_PH$. The additional assumption on $S$ implies that $A_QU_B\Z^d$ also has a nonzero vector inside $A_PH \cap (\R^n \times \mathrm{cone}(S))$.\footnote{This implication uses the fact that lattices are symmetric with respect to sign change; thus it fails without the said assumption on $S$.} Thus
\begin{equation*}
\Phi(A_QU_B,Q^\frac{1}{m}\vv) \leq rP^\frac{d}{mn}
\end{equation*}
for some $\vv \in K(B;Q)$ and $r > 0$ depending on $p$. Minkowski's bound on $\lambda_1$ implies that $P^\frac{d}{mn} \leq \sqrt{d}^\frac{d}{m}$, yielding the desired bound. This concludes the proof of Theorem \ref{thm:main1}.

\subsection{Proof of Corollary}

The condition \eqref{eq:diophantine} is equivalent to
\begin{equation*}
\begin{pmatrix} \vn \\ \vm \end{pmatrix} \in
\begin{pmatrix} \mathrm{Id}_n & \\ -B & \mathrm{Id}_m \end{pmatrix}
\begin{pmatrix} NQ^\frac{1}{n}/2 & \\ & N(\varphi(Q))^\frac{1}{m} \end{pmatrix}
\cdot [-1,1]^d.
\end{equation*}
Therefore, a necessary and sufficient condition for \eqref{eq:diophantine} to have a solution in every congruence class of $\Z^d$ modulo $N$ is that
\begin{equation*}
\begin{pmatrix} \mathrm{Id}_n & \\ -B & \mathrm{Id}_m \end{pmatrix}
\begin{pmatrix} NQ^\frac{1}{n}/2 & \\ & N(\varphi(Q))^\frac{1}{m} \end{pmatrix}
\cdot [-1,1]^d \mbox{ (mod $N\Z^d$)}
\end{equation*}
contains every representative of $\Z^d/N\Z^d$. Normalizing, this is equivalent to the statement that 
\begin{equation} \label{eq:cdn}
\begin{pmatrix} \mathrm{Id}_n & \\ -B & \mathrm{Id}_m \end{pmatrix}
\begin{pmatrix} Q^\frac{1}{n}/2 & \\ & (\varphi(Q))^\frac{1}{m} \end{pmatrix}
\cdot [-1,1]^d \mbox{ (mod $\Z^d$)}
\end{equation}
contains every representative of $(N^{-1}\Z)^d/\Z^d$. We claim the following.
\begin{proposition} \label{prop:cdn}
\eqref{eq:cdn} contains every representative of $(N^{-1}\Z)^d/\Z^d$ for every integer $N > 0$ if and only if
\begin{equation} \label{eq:oink}
B\left( (0, Q^\frac{1}{n}]^n \cap \Z^n \right)
+ [(-\varphi(Q))^\frac{1}{m}, \varphi(Q))^\frac{1}{m}]^m \mbox{ (mod $\Z^m$)}
\end{equation}
covers all of $\R^m/\Z^m$. Furthermore, if \eqref{eq:oink} does not cover $\R^m/\Z^m$, then \eqref{eq:cdn} fails to contain every representative of $(N^{-1}\Z)^d/\Z^d$ for any sufficiently large $N$.
\end{proposition}
\begin{proof}
Suppose first that \eqref{eq:oink} covers all of $\R^m/\Z^m$. We claim that \eqref{eq:cdn} covers $\R^d/\Z^d$, which proves the ``if'' part. Fix a representative $\vx \in \R^n/\Z^n$. Then the ``fiber'' $\mathcal F_\vx$ of \eqref{eq:cdn} over $\vx$, that is, the set of all $\vy \in \R^m/\Z^m$ such that $(\vx,\vy)$ is contained in $\eqref{eq:cdn}$ is given by
\begin{equation*}
\mathcal F_\vx = -B\left( [-Q^\frac{1}{n}/2, Q^\frac{1}{n}/2]^n \cap (\vx + \Z^n) \right)
+ [(-\varphi(Q))^\frac{1}{m}, \varphi(Q))^\frac{1}{m}]^m \mbox{ (mod $\Z^m$)}.
\end{equation*}
On the other hand,
\begin{align*}
&(0,Q^\frac{1}{n}]^n \cap \Z^n \\
&= (-Q^\frac{1}{n}/2, Q^\frac{1}{n}/2]^n \cap (\Z - Q^\frac{1}{n}/2)^n + (Q^\frac{1}{n}/2, \ldots, Q^\frac{1}{n}/2) \\
&\subseteq (-Q^\frac{1}{n}/2, Q^\frac{1}{n}/2]^n \cap (\vx + \Z^n) + (Q^\frac{1}{n}/2, \ldots, Q^\frac{1}{n}/2) - \vx',
\end{align*}
where $\vx'$ is the representative of $(Q^\frac{1}{n}/2, \ldots, Q^\frac{1}{n}/2) + \vx \mbox{ (mod $1$)}$ in $(-1,0]^n$. Hence $-\mathcal F_\vx$ contains a translate of \eqref{eq:oink}, proving the claim. In fact, they are identical provided $\vx' \in (-1 + Q^\frac{1}{n}/2 - \lfloor Q^\frac{1}{n}/2 \rfloor, 0)^n$ --- this will be used to prove the ``only if'' part below.

Thus suppose this time that \eqref{eq:oink} does not cover all of $\R^m/\Z^m$; the uncovered part must have nonempty interior. Moreover, by the remark just made earlier, there exists an open subset $U \subseteq \R^n$ such that, for all $\vx \in U$, $\mathcal F_\vx$ is a translate of \eqref{eq:oink} by a vector depending continuously on $\vx$. This implies that \eqref{eq:cdn} leaves an open subset of $\R^d/\Z^d$ uncovered. Thus for all sufficiently large $N$, it leaves some elements of $(N^{-1}\Z)^d/\Z^d$ uncontained. This completes the proof.
\end{proof}

It is clear that \eqref{eq:oink} covers $\R^m/\Z^m$ if and only if $\varphi(Q)^\frac{1}{m} \geq \sup_{\vv \in \R^m} \gamma_{(B;Q)}(\vv)$ for $p = \infty$ and $S \subseteq \R^m$ the entire unit sphere. Therefore Theorem \ref{thm:main1} and Proposition \ref{prop:cdn} together imply the Corollary.


\newpage

\begin{thebibliography}{99}


\bibitem{AGY21} M. Alam, A. Ghosh, S. Yu. Quantitative Diophantine approximation with congruence conditions. J. Th\'eor. Nr. Bordx. 33 (2021), 261-271.

\bibitem{Beck} J. Beck. Probabilistic Diophantine approximation: I. Kronecker sequences. Ann. Math. 140 (1994), 451-502.

\bibitem{DH23} A. Das, A. Haynes. A three gap theorem for adeles. Ramanujan J. 60 (2023), 13-25.

\bibitem{HM20} A. Haynes, J. Marklof. Higher dimensional Steinhaus and Slater problems via homogeneous dynamics. Ann. Sci. \'Ec. Norm. Sup\'er. (4)53 (2020), no.2, 537-557.

\bibitem{HM22} A. Haynes, J. Marklof. A five distance theorem for Kronecker sequences. Int. Math. Res. Not. 24 (2022), 19747-19789.

\bibitem{KS22} S. Kim, M. Skenderi. Higher-rank pointwise discrepancy bounds and logarithm laws for generic lattices. Acta Arith. 205 (2022), 227-249.

\bibitem{KM98} D. Kleinbock, G. Margulis. Logarithm laws for flows on homogeneous spaces. Inv. Math. 138 (1999),
451-494.

\bibitem{LLL} A. Lenstra, H. Lenstra, L. Lov\'asz. Factoring polynomials with rational coefficients. Math. Ann. 261 (1982), no. 4, 515-534.

\bibitem{MSthree} J. Marklof, A. Str\"ombergsson. The three gap theorem and the space of lattices. Amer. Math. Monthly 124 (2017), no. 8, 741-745.

\bibitem{NRS20} E. Nesharim, R. R\"uhr, R. Shi. Metric Diophantine approximation with congruence conditions. Int. J. Number Theory 2020 16:1923-1933.

\bibitem{Pre15} T. Prest. Gaussian sampling in lattice-based cryptography. Ph.D. thesis, \'Ecole Normale Sup\'erieure. 2015. Available at \url{https://tprest.github.io/pdf/pub/thesis-thomas-prest.pdf}.

\bibitem{Str11} A. Str\"ombergsson. On the probability of a random lattice avoiding a large convex set. Proc. Lond. Math. Soc. (3) 103 (2011), no. 6, 950-1006.


\end{thebibliography}
\end{document}